\documentclass{amsart}
\title{Not every pseudoalgebra is equivalent to a strict one}
\author{Michael A.\ Shulman}
\address{Department of Mathematics\\University of California San Diego\\
  9500 Gilman Dr. \#0112\\San Diego, CA 92093-0112, U.S.A.}
\thanks{The author was supported by a National Science Foundation
  postdoctoral fellowship during the writing of this paper.}
\email{mshulman@ucsd.edu}
\date{\today}

\usepackage{amssymb,amsmath,stmaryrd,mathrsfs}
\usepackage{amsthm}
\usepackage[all,2cell]{xy}
\UseAllTwocells
\usepackage[neveradjust]{paralist}
\usepackage{hyperref}
\usepackage{mathtools}          % for all sorts of things
\usepackage{xspace}
\makeatletter
\let\ea\expandafter
\def\mdef#1#2{\ea\ea\ea\gdef\ea\ea\noexpand#1\ea{\ea\ensuremath\ea{#2}\xspace}}
\def\alwaysmath#1{\ea\ea\ea\global\ea\ea\ea\let\ea\ea\csname your@#1\endcsname\csname #1\endcsname
  \ea\def\csname #1\endcsname{\ensuremath{\csname your@#1\endcsname}}}
\newcount\foreachcount
\def\foreachLetter#1#2#3{\foreachcount=#1
  \ea\loop\ea\ea\ea#3\@Alph\foreachcount
  \advance\foreachcount by 1
  \ifnum\foreachcount<#2\repeat}
\def\definescr#1{\ea\gdef\csname s#1\endcsname{\ensuremath{\mathscr{#1}}\xspace}}
\foreachLetter{1}{27}{\definescr}
\def\definecal#1{\ea\gdef\csname c#1\endcsname{\ensuremath{\mathcal{#1}}\xspace}}
\foreachLetter{1}{27}{\definecal}
\def\definebold#1{\ea\gdef\csname b#1\endcsname{\ensuremath{\mathbf{#1}}\xspace}}
\foreachLetter{1}{27}{\definebold}
\def\autofmt@b#1\autofmt@end{\mathbf{#1}}
\def\autofmt@c#1#2\autofmt@end{\mathcal{#1}\mathit{#2}}
\def\autofmt@s#1#2\autofmt@end{\mathscr{#1}\mathit{#2}}
\def\auto@drop#1{}
\def\autodef#1{\ea\ea\ea\@autodef\ea\ea\ea#1\ea\auto@drop\string#1\autodef@end}
\def\@autodef#1#2#3\autodef@end{%
  \ea\def\ea#1\ea{\ea\ensuremath\ea{\csname autofmt@#2\endcsname#3\autofmt@end}\xspace}}
\def\autodefs@end{blarg!}
\def\autodefs#1{\@autodefs#1\autodefs@end}
\def\@autodefs#1{\ifx#1\autodefs@end%
  \def\autodefs@next{}%
  \else%
  \def\autodefs@next{\autodef#1\@autodefs}%
  \fi\autodefs@next}
\def\defthm#1#2{%
  \newtheorem{#1}{#2}[section]%
  \expandafter\def\csname #1autorefname\endcsname{#2}%
  \expandafter\let\csname c@#1\endcsname\c@thm}
\newtheorem{thm}{Theorem}[section]

\defthm{cor}{Corollary}
\defthm{prop}{Proposition}
\defthm{lem}{Lemma}
\theoremstyle{definition}
\defthm{defn}{Definition}
\theoremstyle{remark}
\defthm{rmk}{Remark}
\defthm{eg}{Example}
\let\c@equation\c@thm
\numberwithin{equation}{section}
\@ifpackageloaded{mathtools}{\mathtoolsset{showonlyrefs,showmanualtags}}{}
\alwaysmath{alpha}
\let\al\alpha
\alwaysmath{lambda}

\alwaysmath{Gamma}
\let\Gm\Gamma
\let\iso\cong
\alwaysmath{otimes}
\mdef{\ten}{\otimes}
\newcommand{\too}[1][]{\ensuremath{\overset{#1}{\longrightarrow}}}
\let\toto\rightrightarrows
\makeatother

\autodefs{\cCat\cCAT\cBicat\cPs\cAlg\cMnd\cTricat\bSet\cUBicat\bSet\nob\cPsmon\cMon}
\def\cat#1{\ensuremath{#1\text{-}\cCat\xspace}}

\def\CAT#1{\ensuremath{#1\text{-}\cCAT\xspace}}
\def\Alg#1{\ensuremath{#1\text{-}\cAlg}}
\def\PsAlg#1{\ensuremath{\cPs\text{-}#1\text{-}\cAlg}}
\def\pscat#1{\ensuremath{\cPs\text{-}#1\text{-}\cCat}}
\mdef\vcatidc{\CAT{\cV}_{\otimes\Sigma,c}}
\mdef\twocatidc{\CAT{2}_{\times\Sigma,c}}
\mdef\cGRAY{\mathfrak{G}\mathbf{ray}}
\mdef\grayidc{\cGRAY_{\times\Sigma,c}}
\mdef\cMndco{\cMnd_{c}}
\mdef\ctwoCat{2\text{-}\cCat}
\alwaysmath{times}
\begin{document}
\maketitle

\begin{abstract}
  We describe a finitary 2-monad on a locally finitely presentable 2-category for which not every pseudoalgebra is equivalent to a strict one.
  This shows that having rank is not a sufficient condition on a 2-monad for every pseudoalgebra to be strictifiable.
  Our counterexample comes from higher category theory: the strict algebras are strict 3-categories, and the pseudoalgebras are a type of semi-strict 3-category lying in between Gray-categories and tricategories.
  Thus, the result follows from the fact that not every Gray-category is equivalent to a strict 3-category, connecting 2-categorical and higher-categorical coherence theory.
  In particular, any nontrivially braided monoidal category gives an example of a pseudoalgebra that is not equivalent to a strict one.
\end{abstract}

\section{Introduction}
\label{sec:intro}

This paper is concerned with theorems of the form ``every weak
structure of some sort is equivalent to a stricter one.''  Theorems of
this sort are sometimes called ``coherence theorems,'' although that
descriptor also often refers to a distinct sort of theorem (one which
explicitly describes the equations that hold in a free structure).
For example, the prototype ``strictification'' theorem is Mac Lane's
result that every monoidal category is equivalent to a strict monoidal
category.  It is natural to look for general contexts in which to
state and prove such theorems, rather than dealing with each case
separately, and one such context is the theory of \emph{2-monads}
initiated in~\cite{bkp:2dmonads} (see also~\cite[\S4]{steve:companion}
for a good introduction).

For a 2-monad $T$ we can construct both the 2-category $\Alg{T}_s$ of
strict algebras and strict morphisms, which satisfy the algebra laws
strictly, and the 2-category $\PsAlg{T}$ of \emph{pseudo} $T$-algebras
and \emph{pseudo} $T$-morphisms, which satisfy the corresponding laws
only up to specified coherent isomorphism.  A natural candidate for a
``general coherence theorem'' would therefore have the form ``for all
2-monads $T$ with some property, every pseudo $T$-algebra is
equivalent to a strict one.''  In particular, there is a 2-monad $S$
on \cCat for which $S$-algebras are strict monoidal categories, while
pseudo $S$-algebras are, essentially, non-strict monoidal categories;
thus Mac Lane's coherence theorem can be regarded as having this form.

\begin{rmk}\label{rmk:bias}
  There is a subtlety here, however: pseudo $S$-algebras are actually
  ``unbiased'' monoidal categories, which have a basic $n$-ary tensor
  product for all $n\ge 0$, rather than merely binary and nullary
  operations as in the usual presentation.  The 2-category $\PsAlg{S}$
  turns out to be equivalent to the usual 2-category of ``biased''
  monoidal categories and strong monoidal functors, but it is
  \emph{this} equivalence where the hard work in Mac Lane's theorem
  really lies.  The fact that every pseudo $S$-algebra is equivalent
  to a strict one is much easier, by comparison, and in fact follows
  from the general coherence theorems mentioned below.

  This sort of situation is quite common in the study of coherence.
  One possible reaction is to say that pseudoalgebras are not really
  the objects of interest, but are of mainly technical usefulness.
  Another point of view is that pseudoalgebras and other ``unbiased''
  structures are really the fundamental objects, with the more usual
  sort of ``biased'' definitions only being correct insofar as they
  are a more economical presentation of an unbiased one.  But the
  question of strictifying pseudoalgebras, which we address here, is
  of interest in either case.
\end{rmk}

One of the first and most general strictification theorems for pseudoalgebras was proven in~\cite{power:coherence}, under the hypothesis that the 2-monad in question preserves a suitable factorization system.
A slight refinement of this, along with some other sufficient conditions regarding the preservation of certain 2-categorical colimits, can be found in~\cite{lack:codescent-coh}.
There are 2-monads for which not every pseudoalgebra is equivalent to a strict one, such as that
in~\cite[Example 3.1]{lack:codescent-coh}, but until now all known
such examples have been fairly contrived and lived on poorly behaved
2-categories, suggesting a conjecture that the theorem might always
hold in well-behaved cases.  The purpose of this paper is to describe
a very natural and otherwise well-behaved 2-monad on a well-behaved
2-category for which the ``coherence theorem'' fails.  In particular,
the 2-category in question is locally finitely presentable, and the
2-monad is finitary (preserves filtered colimits).

Of course, there are many known situations in which not every weak
structure is equivalent to a strict one.  For instance, not every
symmetric monoidal category is equivalent to a strictly-symmetric
strict monoidal category.  However, this is not an instance of the
notion of pseudoalgebra over a 2-monad.  There is a 2-monad whose
strict algebras are strictly-symmetric strict monoidal categories, but
its pseudoalgebras cannot be identified with non-strict symmetric
monoidal categories.  Instead, non-strict symmetric monoidal
categories are the pseudoalgebras for a 2-monad whose strict algebras
are non-strictly-symmetric strict monoidal categories, and for this
2-monad the coherence theorem does hold.

For our counterexample, we exploit a related situation, namely the
fact that not every tricategory is equivalent to a strict 3-category.
Since this situation is ``higher-dimensional,'' it may at first not
seem to fall within the realm of 2-monad theory.
However, it has emerged recently (see for instance~\cite{lack:icons,gg:ldstr-tricat}) that by using special sorts of higher transformations, one can construct ``low-dimensional categories of higher-dimensional categories.''
In this spirit, we will show that there is a 2-monad $T_{\cCat}$ on the 2-category of ``\cCat-enriched 2-graphs,''
whose strict algebras are strict 3-categories, and whose
pseudoalgebras are a type of ``semi-strict'' 3-category.  We call
these \emph{iconic tricategories}, since they can be identified with
tricategories whose associativity and unit constraints are
\emph{icons} in the sense of~\cite{lack:icons}, i.e.\ have identity
1-cell components.  The main theorem follows once we observe that all
Gray-categories are iconic, so that every tricategory is equivalent to an iconic one; thus not all iconic tricategories can be
equivalent to strict 3-categories.  We can also give a more direct
proof by restricting to doubly-degenerate objects, appealing instead
to the fact that not every braided monoidal category is equivalent to
a strictly symmetric one.

The 2-monad $T_{\cCat}$ can be described very explicitly, but identifying its pseudoalgebras is easier if we also derive it from some abstract machinery.
As observed in~\cite{leinster:higher-opds,cheng:cmp-opd-ncats,bcw:algop-enr-ii}, we can construct monads whose algebras are enriched $n$-categories by iteratively splicing together monads whose algebras are enriched 1-categories, using distributive laws.
By identifying strict 3-categories with \cCat-enriched 2-categories, we can obtain the 2-monad $T_{\cCat}$ by one application of this procedure, as long as we carry a \cCat-enrichment through the construction so as to obtain a 2-monad instead of an ordinary one.
Thus, a large part of the paper is spent setting up this machinery in the enriched setting.

In \S\ref{sec:2monads} we recall the basic notions of 2-monad theory and the general coherence theorems of~\cite{power:3dmonads,lack:codescent-coh}.
Then in \S\ref{sec:graphs} we describe the general construction of a \cV-monad $\Gm_{\cW}$ whose algebras are \cW-enriched categories, for any bicomplete cartesian closed category \cV and any monoidal \cV-category \cW.
(Our primary interest is in the case $\cV=\cCat$, but the greater generality clarifies the exposition.)
In \S\ref{sec:pscat} we remark on the application of the coherence theorems to $\Gm_{\cW}$ when $\cV=\cCat$, yielding the strictification theorem for ``(enriched) unbiased bicategories.''
Then in \S\ref{sec:iterated} we describe the iteration procedure as in the references above, but carrying through an ambient enrichment over any \cV as in \S\ref{sec:graphs}, thereby yielding a \cV-monad $T_\cW$ whose algebras are \cW-enriched 2-categories.
Finally, in \S\ref{sec:enrich-bicat} we take $\cV=\cW=\cCat$, identify pseudo $T_{\cCat}$-algebras with iconic tricategories, and conclude that not every pseudo $T_{\cCat}$-algebra is equivalent to a strict one.

I would like to thank Steve Lack for a careful reading of early drafts
of this paper and several very helpful suggestions.

\section{Strictification of pseudoalgebras}
\label{sec:2monads}

We begin by briefly reviewing the basic notions of 2-monad theory and the general coherence theorem.
By a \textbf{2-monad} we will always mean a \emph{strict} 2-monad; that is, a \cCat-enriched monad.
From general enriched category theory (see for instance~\cite{kelly:enriched}), any such 2-monad $(T,\mu,\eta)$ has a 2-category of algebras, which is denoted by $\Alg{T}_s$.
Its objects are pairs $(A,a)$, where $a\colon TA\to A$ satisfies $a \circ \eta = 1$ and $a \circ Ta = a\circ \mu$ exactly, and a morphism $(A,a)\to (B,b)$ is a morphism $f\colon A\to B$ such that $f \circ a = b\circ Tf$ exactly.
We call these \emph{strict $T$-algebras} and \emph{strict $T$-morphisms}.

We also have the 2-category $\PsAlg{T}$, whose objects and morphisms are \emph{pseudo $T$-algebras} and \emph{pseudo $T$-morphisms}, respectively.
A pseudo $T$-algebra consists of $A$ and $a\colon TA\to A$ together with isomorphisms $a \circ \eta \cong 1$ and $a \circ Ta \cong a\circ \mu$, satisfying appropriate coherence laws.
Similarly, a pseudo $T$-morphism is $f\colon A\to B$ together with an isomorphism $b\circ Tf \cong f \circ a$ satisfying appropriate axioms.
(If the isomorphism is replaced by a not-necessarily-invertible morphism $b\circ Tf \to f \circ a$ or $f \circ a \to b\circ Tf$, we call it a \emph{lax} or \emph{colax} $T$-morphism, respectively.)
There is an obvious inclusion $\Alg{T}_s \hookrightarrow \PsAlg{T}$, and the question of strictification is whether it is essentially surjective (up to equivalence).

In~\cite{power:coherence}, Power proved a general strictification theorem in the following situation.
We suppose that the base 2-category \cK\ has a factorization system $(\sE,\sM)$ which is \emph{enhanced}, meaning that given any isomorphism $\alpha\colon t e \xrightarrow{\cong} m s$ with $e\in \sE$ and $m\in \sM$, there exists a unique pair $(r,\beta)$ with $r e=s$, $\beta\colon t \xrightarrow{\cong} m r$, and $\beta e = \alpha$.
The prototypical example is (bijective on objects, fully faithful) on \cCat.
We suppose furthermore that if $j\in\sM$ and $jk\cong 1$, then $kj\cong 1$.
(This is the case whenever all morphisms in \sM\ are representably fully faithful, i.e.\ $\cK(X,A) \xrightarrow{\cK(X,j)} \cK(X,B)$ is fully faithful for all $j\colon A\to B$ in \sM.)
Power showed (essentially) that under these hypotheses, if $T$ is a 2-monad on \cK\ which preserves \sE-morphisms, then every pseudo $T$-algebra is equivalent to a strict one.

In~\cite{lack:codescent-coh}, Lack observed that Power's hypotheses actually imply that $\Alg{T}_s \hookrightarrow \PsAlg{T}$ has a left 2-adjoint, and the components of the adjunction unit are equivalences.
Thus, not only is every pseudo $T$-algebra equivalent to a strict one, but in a certain canonical universally characterized way.
He also noted that such a left adjoint exists as soon as $\Alg{T}_s$ has a certain type of \cCat-enriched colimit called a \emph{reflexive codescent object}.
Two natural hypotheses under which $\Alg{T}_s$ has reflexive codescent objects are (1) \cK\ has reflexive codescent objects and $T$ preserves them, or (2) \cK\ is locally presentable and $T$ is accessible (has a rank).
Lack proved that hypothesis (1) also implies the strictification theorem (i.e.\ the components of the unit are equivalences).
The example we will discuss shows that hypothesis (2) does not.

\section{Enriched graphs and categories}
\label{sec:graphs}

We now describe a monad whose algebras are categories enriched over some monoidal category \cW.
This is well-known; the only slight novelty is the observation that when \cW\ is a monoidal 2-category, the monad is a 2-monad.
There is not much special about \cCat-enrichment in this observation: when \cW\ is a monoidal \cV-category, for any complete and cocomplete
cartesian closed category \cV, the resulting monad is a \cV-monad.
(We do need \cV\ to be \emph{cartesian} monoidal, however.)
In fact, replacing \cCat by \cV can even make things clearer, since it avoids confusion between the categories we are defining a monad for and the categories we are enriching over.

Thus, for this section, let \cV be complete, cocomplete, and cartesian closed; in the next section we will specialize to $\cV=\cCat$.
For any \cV-category \cW, a \textbf{\cW-graph} consists of a set $A_0$
along with, for every $x,y\in A_0$, an object $A(x,y)\in\cW$.  We
define a \cV-category $\cG(\cW)$ of \cW-graphs, with hom-objects
\[\cG(\cW)(A,B) = \sum_{f_0\colon A_0\to B_0} \;\prod_{x,y\in A_0}
\cW\Big(A(x,y),\;B(f_0(x),f_0(y))\Big)
\]
If the terminal object $1$ of \cV is indecomposable (i.e.\ $\cV(1,-)$
preserves sums), then a morphism $f\colon A\to B$ in the underlying
ordinary category of $\cG(\cW)$ consists of a function $f_0\colon
A_0\to B_0$ together with, for every $x,y\in A_0$, a morphism
$A(x,y)\to B(f_0(x),f_0(y))$ in \cW.

Any \cV-functor $F\colon \cW\to\cW'$ induces a \cV-functor
$\cG(F)\colon \cG(\cW)\to\cG(\cW')$, which leaves the sets $A_0$ unchanged and applies $F$
on hom-objects.  Likewise, any \cV-transformation $\al\colon F\to G$
induces $\cG(\al)\colon \cG(F)\to\cG(G)$, defined by the map
\[1 \too \sum_{f_0\colon A_0\to A_0} \;\prod_{x,y\in A_0}
\cW'\Big(F(A(x,y)),\;G(A(x,y))\Big)
\]
which is determined by components of $\al$ mapping into the summand
$f_0 = \mathrm{id}$.  Thus, \cG defines an endo-2-functor of the
2-category \CAT{\cV} of \cV-categories.  In the case $\cV=\bSet$, this
is the functor of the same name from~\cite[\S2.1]{bcw:algop-enr-ii}.

Now suppose that \cW is a monoidal \cV-category, so that we can also
consider \cW-enriched categories.  We can then define a \cV-category
\cat{\cW} of small \cW-categories, whose hom-object $\cat{\cW}(A,B)$
is an equalizer of the following form:
\begin{multline*}
  \vcenter{\xymatrix{ 
\displaystyle \sum_{f_0\colon A_0\to B_0}^{\phantom{A}} \prod_{x,y\in A_0}
\cW\Big(A(x,y),\; B(f_0(x),f_0(y))\Big)
\ar@<1mm>[r]\ar@<-1mm>[r] &}}\\
\sum_{f_0\colon A_0\to B_0}^{\phantom{A}} \prod_{x,y,z\in A_0}
\cW\Big(A(y,z)\otimes A(x,y),\;B(f_0(x),f_0(z))\Big)
\end{multline*}

The assumption that \cV is cartesian, rather than merely symmetric
monoidal, is essential in defining one of these two morphisms.  If $1$
is indecomposable in \cV, then a morphism $f\colon A\to B$ in the
underlying ordinary category of $\cat{\cW}$ is exactly a \cW-enriched
functor in the usual sense.

\begin{eg}
  Since \cV is a monoidal \cV-category, we have in particular a \cV-category
  \cat{\cV} of small \cV-categories.  It is well-known that \cat{\cV}
  is also closed symmetric monoidal and hence enriched over itself,
  but it is not as commonly observed that it can be enriched over \cV
  as well.
  As we will see in the next section, however, the enrichment is not necessarily what one would expect.
\end{eg}

% Both maps fix the $f_0$ component, and thus are each determined by a
% family of $(x',y',z')$-components.  The first map has
% $(x',y',z')$-component
% \begin{align*}
%   \prod_{x,y\in A_0} B(f_0(x),f_0(y))^{A(x,y)}
%   &\too B(f_0(y'),f_0(z'))^{A(y',z')} \times B(f_0(x'),f_0(y'))^{A(x',y')}\\
%   &\too \Big(B(f_0(y'),f_0(z')) \times B(f_0(x'),f_0(y'))\Big)^{A(y',z')\times A(x',y')}\\
%   &\too B(f_0(x'),f_0(z'))^{A(y',z')\times A(x',y')}
% \end{align*}
% and the second has $(x',y',z')$-component
% \begin{align*}
%   \prod_{x,y\in A_0} B(f_0(x),f_0(y))^{A(x,y)}
%   &\too B(f_0(x'),f_0(z'))^{A(x',z')}\\
%   &\too B(f_0(x'),f_0(z'))^{A(y',z')\times A(x',y')}
% \end{align*}

There is an evident forgetful \cV-functor $\cU_\cW\colon
\cat{\cW}\to\cG(\cW)$.  If we suppose in addition that \cW is
\textbf{\ten-distributive}, i.e.\ it has small sums which are
preserved on both sides by $\otimes$, then $\cU_\cW$ has a left
adjoint and is monadic.  Its left adjoint $\cF_\cW$ acts as the identity on
$A_0$, with
\[ \cF_\cW(A)(x,y) = \sum_{z_1,\dots,z_n} A(z_n,y) \otimes\dots \otimes A(x,z_1).
\]
(Again, we need \cV to be cartesian to make $\cF_\cW$ into a
\cV-functor.)  The sum always includes $n=0$, in which case the term
is $A(x,y)$, and if $x=y$ it also includes ``$n=-1$'' whose
corresponding term is the unit object of \cW.  Preservation of sums by
tensor products in \cW enables us to make this into a \cW-category,
and its universal property is easy to verify.
Following~\cite[\S4]{bcw:algop-enr-ii}, we write $\Gm_\cW$ for the
associated \cV-monad on $\cG(\cW)$, whose algebras are \cW-enriched
categories.

\section{Pseudo enriched categories}
\label{sec:pscat}

We now specialize to the case $\cV=\cCat$.  Thus, for any 2-category
\cW, we have a 2-category $\cG(\cW)$ of \cW-graphs.  Its objects and
morphisms are what one would expect, while a 2-cell $\al\colon f\to g$
between morphisms $f,g\colon A\to B$ of \cW-graphs consists of
\begin{enumerate}
\item The assertion that $f_0=g_0$, and
\item For each $x,y\in A_0$, a 2-cell $A(x,y) \!\xymatrix@C=3pc{
    \rtwocell^f_g{\al} & }\! B(f_0x, f_0y)$ in \cW.
\end{enumerate}
If \cW is moreover a \ten-distributive monoidal 2-category, then we
have a 2-monad $\Gm_\cW$ on $\cG(\cW)$ such that strict
$\Gm_\cW$-algebras are small \cW-enriched categories, and strict
$\Gm_\cW$-morphisms are \cW-enriched functors.  A
$\Gm_\cW$-transformation $\al\colon f\to g$ between such functors
consists of a 2-cell of \cW-graphs, as above, such that
\begin{enumerate}\setcounter{enumi}{2}
\item For each $x,y,z\in A$, we have
  \begin{multline*}
    \vcenter{\xymatrix@C=6pc{A(y,z)\otimes A(x,y)
        \rtwocell^{f\otimes f}_{g\otimes g}{\al^{y,z}\otimes \al^{x,y}\hspace{2cm}}\ar[d]
        % \drtwocell\omit{<2>\iso}
        &
        B(f_0 y,f_0 z)\otimes  B(f_0 x, f_0 y) \ar[d]\\
        A(x,z) \ar[r] &  B(f_0x, f_0 z)}} =\\
    \vcenter{\xymatrix@C=3pc{A(y,z)\otimes A(x,y)
        \ar[r]^-{f\otimes f}
        \ar[d] %\drtwocell\omit{<-2>\iso}
        &
        B(f_0 y,f_0 z)\otimes  B(f_0 x, f_0 y) \ar[d]\\
        A(x,z) \rtwocell^f_g{\al^{x,z}\hspace{1.3cm}} &  B(f_0x, f_0 z)}}
  \end{multline*}
\item For each $x\in A$, we have
  \[\vcenter{\xymatrix@C=3pc@R=3pc{
      1 \ar[r] \ar[dr]
      % \drlowertwocell{<-.7>\iso}
      &
      A(x,x) \dtwocell<5>^f_g{\al^{x,x}} \\
      &  B(f_0 x, f_0 x)
    }}
  \quad=\quad
  \vcenter{\xymatrix{
      1 \ar[r] \ar[dr]
      % \drlowertwocell{<-.7>\iso}
      &
      A(x,x) \ar[d]^f \\
      &  B(f_0 x, f_0 x)
    }}
  \]
\end{enumerate}
We call such a 2-cell a \textbf{\cW-icon}.  In the case $\cW=\cCat$,
the 2-monad $\Gm_{\cCat}$ is the same one considered
in~\cite[\S6.2]{lack:icons} and in~\cite{lp:2nerves}, and a \cCat-icon
is the same as an icon in the sense defined there.
The word ``icon'' is an acronym for ``Identity Component Oplax Natural transformation,'' since icons can be identified with oplax transformations whose 1-morphism components are identities.

\begin{rmk}
The 2-category $\Alg{\Gm_{\cCat}}$ of 2-categories, 2-functors, and icons is the prototypical ``low-dimensional category of higher-dimensional categories.''
Normally, of course, we regard 2-categories as forming a (strict or weak) 3-category, with pseudonatural transformations and modifications as the 2- and 3-morphisms.
The important insight is that by restricting the transformations to be have identity components, we can allow them to be otherwise oplax (not just pseudo), and we can moreover discard the modifications and obtain a well-behaved 2-category.
\end{rmk}

Now, returning to the case of general \cW, we can also consider
\emph{pseudo} $\Gm_\cW$-algebras, which we call \textbf{unbiased
  pseudo \cW-categories}.  Inspecting the monad $\Gm_{\cCat}$, we see that a pseudo $\Gm_\cW$-algebra has a set of objects $A_0$,
hom-objects $A(x,y)\in\cW$, and basic $n$-ary composition operations
\[A(x_{n-1},x_n) \otimes\cdots\otimes A(x_0,x_1) \too A(x_0,x_n)
\]
for all $n\ge 0$, along with unbiased associativity isomorphisms
satisfying coherence laws.  In particular, when $\cW=\cCat$ we
obtain \emph{unbiased bicategories}.  We have corresponding notions of
\textbf{\cW-pseudofunctors} and \textbf{icons} between pseudo
\cW-categories, forming the 2-category $\PsAlg{\Gm_\cW}$.  The
definition of \cW-icon for pseudo \cW-categories is just like that
above for strict ones, except that appropriate isomorphisms must be
inserted in previously commutative squares and triangles.

One can also define a notion of \emph{biased pseudo \cW-category}, by
simply writing out the usual definition of bicategory and replacing
all categories, functors, and transformations by objects, morphisms,
and 2-cells in \cW (and the cartesian product of categories by the
tensor product in \cW).  Similarly, one can define \cW-pseudofunctors
and \cW-icons between these, forming a 2-category \pscat{\cW}.  We
then have:

\begin{lem}\label{thm:bias}
  The 2-categories $\PsAlg{\Gm_\cW}$ and $\pscat{\cW}$ are
  2-equivalent.
\end{lem}
\begin{proof}
  This is basically identical to the corresponding result for bicategories
  or monoidal categories, see
  e.g.~\cite[3.2.4]{leinster:higher-opds}.  (However, recall \autoref{rmk:bias}.)
\end{proof}

Accordingly, we will write \pscat{\cW} for this 2-category and call
its objects simply \emph{pseudo \cW-categories}.  (Note, though, that
the lemma would be false if we used \emph{strict} functors in defining
these 2-categories instead of pseudo ones.)  Of course,
$\pscat{\cCat}\simeq \cBicat$ is the 2-category of bicategories,
pseudofunctors, and icons.

Note that a pseudo \cW-category with one object is precisely a
pseudomonoid in \cW, just as a bicategory with one object is a
monoidal category.  Furthermore, in this case \cW-pseudofunctors
reduce to pseudomorphisms of pseudomonoids (such as strong monoidal functors), and \cW-icons to
pseudomonoid transformations (such as monoidal transformations).  (As observed in~\cite{lack:icons}, this
is one of the advantages of icons: other kinds of transformation
between one-object bicategories do not correspond so closely to
monoidal transformations.)  We thus record:

\begin{lem}\label{thm:psmon}
  The 2-category $\cPsmon(\cW)$ of pseudomonoids in \cW embeds
  2-fully-faithfully in \pscat{\cW} as the pseudo \cW-categories with
  one object.
\end{lem}

When \cW is symmetric, so that \cW-categories and \cW-pseudomonoids
have tensor products, then this embedding is also strong monoidal.  Of
course, strict \cW-enriched categories correspond to strict monoids.

%\begin{rmk}\label{thm:pscat-coh}
Although our main theorem will be about an iterated version of $\Gm_\cW$, it is natural to ask whether $\Gm_\cW$ itself satisfies the strictification theorem.
One of the applications of the general coherence theorem given in~\cite{power:coherence} was to unbiased bicategories, but only with a fixed set of objects (i.e.\ working with a different 2-monad on a different 2-category for every set of objects).
The monad $\Gm_\cW$, as we have defined it, does not quite satisfy any of the hypotheses of the strictification theorems cited in \S\ref{sec:2monads}, but Steve Lack has observed that essentially the same proofs can nevertheless be applied as long as we carefully note that the hypotheses are used only in cases where they are valid.

For instance, suppose that \cW has an enhanced factorization system $(\sE,\sM)$ such that \sE is preserved by \ten on both sides and all \sM-maps are representably fully faithful.
(This includes the factorization system (bijective on objects, fully faithful) on \cCat.)
Then $\cG(\cW)$ has an enhanced factorization system $(\cG(\sE),\cG(\sM))$, where the $\cG(\sE)$-maps are bijective on objects and locally (i.e.\ hom-wise) in \sE, and the $\cG(\sM)$-maps are locally in \sM.
The 2-monad $\Gm_{\cW}$ preserves this class $\cG(\sE)$, but the final hypothesis, that $j\in \cG(\sM)$ and $j k\cong 1$ imply $k j\cong 1$, fails.
However, in the coherence theorem we only apply this hypothesis to the $\cG(\sM)$-half of the factorization of the structure map of a pseudo $\Gm_\cW$-algebra, and we claim that such a map is always representably fully faithful (which implies the desired conclusion).
For a map in $\cG(\cW)$ is representably fully faithful just when it is locally representably fully faithful and also injective on objects.
But we have assumed that all \sM-maps are representably fully faithful, and the structure map of a pseudo $\Gm_\cW$-algebra is always bijective on objects, so the map in question must also be so.

Similarly, $\Gm_\cW$ need not preserve all reflexive codescent objects, but if \cW has reflexive codescent objects preserved on either side by \ten, then $\cG(\cW)$ has, and $\Gm_\cW$ preserves, reflexive codescent objects of diagrams whose morphisms are all bijective on objects.
This follows from the ``4-by-4 lemma'' for reflexive codescent objects alluded to in~\cite[Prop.~4.3]{lack:codescent-coh} as a generalization of~\cite[2.1]{klw:refl-coinv}.
This is sufficient to prove the coherence theorem for $\Gm_\cW$, since its multiplication and unit are bijective on objects, as is the structure map of any pseudoalgebra (because it is a retraction of the unit, up to an invertible 2-cell in $\cG(\cW)$, and such a 2-cell requires its domain and codomain to act identically on objects).
In particular, this applies whenever \cW is closed monoidal and cocomplete (such as $\cW=\cCat$).
Thus, for any such \cW, the strictification theorem holds for pseudo \cW-categories.

\section{Monadic iterated enrichment}
\label{sec:iterated}

We now describe how to iterate the construction of the monad $\Gm_\cW$ to obtain a monad whose algebras are enriched 2-categories.
This procedure is described in~\cite[Appendix F]{leinster:higher-opds},~\cite{cheng:cmp-opd-ncats}, and most recently~\cite{bcw:algop-enr-ii}.
As in \S\ref{sec:graphs}, the only novelty is carrying through a \cCat-enrichment to obtain a 2-monad, and the only special property of \cCat required is that it is cartesian monoidal.
Thus, we revert to the situation of a complete and cocomplete cartesian closed category \cV and a monoidal \cV-category \cW, which in this section we additionally assume to be symmetric.

\begin{rmk}
  For our main theorem in \S\ref{sec:enrich-bicat} we will require
  only the case when $\cW$ is cartesian monoidal, but it is not much
  more work to consider the more general symmetric monoidal case.  The
  authors of~\cite{bcw:algop-enr-ii} work in the yet more general
  situation where \cW is only lax monoidal, but in such generality it
  seems that $\Gm_{\cW}$ need not be colax monoidal, as in
  \autoref{thm:g-lift} below.
\end{rmk}

As soon as \cW is symmetric, the \cV-category $\cat{\cW}$ is also symmetric monoidal.
Thus we can consider $(\cat{\cW})$-enriched categories, which it is natural to call \emph{\cW-enriched 2-categories}.
The theory of \S\ref{sec:graphs} shows that \cW-enriched 2-categories are monadic over $(\cat{\cW})$-graphs; our goal is to additionally exhibit them as monadic over \textbf{\cW-enriched 2-graphs}, i.e.\ $\cG(\cW)$-graphs.
The resulting monad will thus simultaneously build in both the ``horizontal'' and ``vertical'' composition operations of a 2-category.
The idea is to construct such a monad by combining two instances of the $\Gm$ monads, one for each composition operation.
The combination happens using the standard method of distributive laws, as in~\cite{beck:dl}.

We begin by observing that \cG is an endo-2-functor of $\CAT{\cV}$.
Moreover, when \cW is monoidal, so is $\cG(\cW)$: we set $(A\ten B)_0 = A_0\times B_0$
with
\[ (A\ten B)\big((x,y),(x',y')\big) = A(x,x') \ten B(y,y').
\]
Similarly, \cG also preserves monoidal functors of any type (strong, lax, colax), monoidal transformations, braidings, and symmetries.
Finally, if \cW is \ten-distributive, then $\cG(\cW)$ has small sums (take the disjoint union of object sets, with initial objects as hom-objects between them) and is \ten-distributive.
Thus, \cG defines an endo-2-functor of any 2-category of monoidal \cV-categories we might desire.

We now want to make the construction of the monad $\Gm_\cW$ functorial
as well.  This is done for both lax and colax monoidal functors
in~\cite{bcw:algop-enr-ii}, but we will focus on the colax case, which
can be iterated more successfully.  Let \vcatidc denote the
2-category of \ten-distributive symmetric monoidal \cV-categories,
colax symmetric monoidal functors that preserve small sums, and
monoidal transformations.  This will be the domain of our monad-valued functor; it is closely related to the 2-category
$\mathrm{OpDISTMULT}$ of~\cite[\S6.5]{bcw:algop-enr-ii}.  Note that
any functor or transformation between cartesian monoidal categories is
automatically colax symmetric monoidal.

The codomain of our monad-valued functor must be a 2-category of monads, and for purposes of iteration we would like it to consist of monads in \vcatidc itself.
Recall from~\cite{street:ftm} that a \emph{monad} in a 2-category \cK is an endo-1-morphism $t\colon A\to A$ together with 2-morphisms $tt\to t$ and $1\to t$ satisfying the usual laws.
Given two such monads $t\colon A\to A$ to $s\colon B\to B$, a \emph{colax monad morphism} between them (called a ``monad opfunctor'' in~\cite{street:ftm}) consists of a 1-morphism $f\colon A \to B$ together with a 2-cell $ft \to sf$ satisfying some axioms.
We write $\cMndco(\cK)$ for the 2-category of monads and colax monad morphisms in \cK.

\begin{lem}\label{thm:g-lift}
  The 2-functor \cG lifts to a 2-functor
  \[\cG\colon \vcatidc \too  \cMndco(\vcatidc),
  \]
  which sends \cW to $(\cG(\cW),\Gm_\cW)$.
\end{lem}
\begin{proof}
  We first need to know that $\Gm_\cW$ is a monad in \vcatidc.  It
  certainly preserves small sums.  A colax monoidal structure for it
  should consist of maps
  \[ \Gm_\cW(A\ten B) \to \Gm_\cW(A) \ten \Gm_\cW(B)
  \]
  for \cW-graphs $A$ and $B$.  Both sides have the same set of objects
  $A_0\times B_0$, so we can take this map to be the identity on
  objects, with hom-morphisms 
  \[ \Gm_\cW(A\ten B)((x,y),(x',y')) \to \Gm_\cW(A)(x,x') \ten \Gm_\cW(B)(y,y')
  \]
  given by the ``rearrangement'' map
  \begin{gather*}
    % Weird error messages doing this all in XY!
    \sum_{(z_1,w_1),\dots,(z_n,w_n)}
    \Big(A(z_{n},x') \otimes B(w_n,y')\Big) \otimes \cdots \otimes
    \Big(A(x,z_1) \otimes B(y,w_1)\Big)\\
    \xymatrix{\ar[d]\\ \ }\\
    \left(\sum_{z_1,\dots,z_m} A(z_m,x')\otimes\dots\otimes A(x,z_1)\right)
    \otimes
    \left(\sum_{w_1,\dots,w_k} B(w_k,y')\otimes\dots\otimes B(y,w_1)\right)
  \end{gather*}
  which maps into the summand where both $m$ and $k$ are equal to
  $n$.  Note that this requires symmetry and associativity of \cW, and
  also that it is not an isomorphism.  It is straightforward to verify
  the necessary axioms; thus $\Gm_\cW$ is a monad in \vcatidc.  (When
  \cW is cartesian, so is $\cG(\cW)$, and so everything is automatic.)

  Next, we show that if $F\colon \cW\to\cW'$ is colax monoidal
  and preserves small sums, then $\cG(F)$ is a colax 
  monad functor from $\Gm_{\cW}$ to $\Gm_{\cW'}$.  That is, we require a
  natural transformation $\cG(F) \circ \Gm_\cW \to \Gm_{\cW'} \circ
  \cG(F)$ satisfying two axioms.  Since all three functors involved
  are the identity on objects, it suffices to give natural maps
  \[ F\Big(\sum_{z_1,\dots,z_n} A(z_n,y) \otimes\dots \otimes A(x,z_1)\Big)
  \too
  \sum_{z_1,\dots,z_n} F(A(z_n,y)) \otimes\dots \otimes F(A(x,z_1))
  \]
  For this we can simply use the colax comparison maps for $F$ along
  with the fact that $F$ preserves sums.  All the axioms are again
  straightforward to verify, as is the final requisite fact that \cG
  takes transformations in \vcatidc to monad 2-cells.
\end{proof}

Thus, for any $\cW\in\vcatidc$, the monad $\Gm_\cW$ on $\cG(\cW)$ is
actually a monad in \vcatidc.  Since 2-functors take monads to monads,
we can then apply \cG again to $\Gm_\cW$ itself, to obtain a new monad
$\cG(\Gm_\cW)$ in $\cMndco(\vcatidc)$ on $(\cG(\cG(\cW)),
\Gm_{\cG(\cW)})$.  As observed in~\cite{street:ftm}, such a monad in a
2-category of monads amounts to a \emph{distributive law} in \vcatidc
in the sense of~\cite{beck:dl}:
\[\lambda\colon \cG(\Gm_\cW) \circ \Gm_{\cG(\cW)} \too 
\Gm_{\cG(\cW)}\circ \cG(\Gm_\cW)
\]
between the monads $\cG(\Gm_\cW)$ and $\Gm_{\cG(\cW)}$ on
$\cG(\cG(\cW))$.

We find it conceptually helpful to write out this distributive law
explicitly, although our proofs will proceed at a high enough level to
make such a description mostly unnecessary.  An object of $\cG(\cG(\cW))$ is
a \textbf{\cW-enriched 2-graph}: it consists of a directed graph $A_1
\toto A_0$, together with an object $A(f,g)$ of \cW for every parallel
pair of edges in $A_1$.  The monad $\cG(\Gm_\cW)$ is the identity on
$A_0$ and $A_1$, with
\[ \cG(\Gm_\cW)(A)(f,g) =
\sum_{h_1,\dots,h_n} A(h_n,g)\ten\cdots \ten A(f,h_1).
\]
The monad $\Gm_{\cG(\cW)}$ acts on $A_1 \toto A_0$ as the free
category monad, with
\[ \Gm_{\cG(\cW)}(A)\big((f_n,\dots,f_1),(g_m,\dots,g_1)\big) =
\begin{cases}
  A(f_n,g_n) \ten \cdots \ten A(f_1,g_1) & \quad \text{if }n=m\\
  \emptyset & \quad \text{if }n\neq m.
\end{cases}
\]
Thus both composites $\cG(\Gm_\cW) \circ \Gm_{\cG(\cW)}$ and
$\Gm_{\cG(\cW)}\circ \cG(\Gm_\cW)$ act as the free category monad on
underlying directed graphs.  For the first, we have
{\small
\begin{multline}\label{eq:rectpdsum}
  \cG(\Gm_\cW)(\Gm_{\cG(\cW)}(A)) \big((f_n,\dots,f_1),(g_n,\dots,g_1)\big) =\\
  \sum_{k,h_{i,j}} \Big(A(h_{n, k},g_n) \ten \cdots\ten A(h_{1, k},g_1)\Big)
  \ten\cdots\ten \Big(A(f_n,h_{n,1}) \ten \cdots\ten A(f_1,h_{1,1})\Big)
\end{multline}}
while for the second, we have
{\small
\begin{multline}\label{eq:pdsum}
  \Gm_{\cG(\cW)}(\cG(\Gm_\cW)(A)) \big((f_n,\dots,f_1),(g_n,\dots,g_1)\big) =\\
  \sum_{k_i,h_{i,j}} \Big( A(h_{n,k_n},g_n) \ten\cdots\ten A(f_n,h_{n,1}) \Big)
  \ten\cdots\ten \Big( A(h_{1,k_1},g_1) \ten\cdots\ten A(f_1,h_{1,1}) \Big).
\end{multline}}

The first sum is over all \emph{rectangular} arrays $(h_{i,j})$ with
$1\le i\le n$ and $1\le j\le k$, while the second sum is over arrays
$(h_{i,j})$ where $1\le i\le n$ and $1\le j\le k_i$, with the bound
$k_i$ possibly depending on $i$.  In other words, the first
corresponds to pasting diagrams of 2-cells in a 2- or 3-category such as the
following:
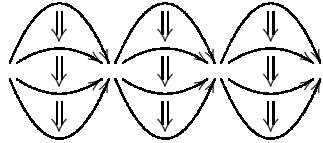
\[\vcenter{\xymatrix{
    \\ \quad
    \ar@/_9mm/[r]^{}="34" \ar@/_3mm/[r]_{}="33b"^{}="33a"
    \ar@/^3mm/[r]_{}="32b"^{}="32a" \ar@/^9mm/[r]_{}="31"
    \ar@{=>}"31";"32a"
    \ar@{=>}"32b";"33a"
    \ar@{=>}"33b";"34"
    & \quad
    \ar@/_9mm/[r]^{}="24" \ar@/_3mm/[r]_{}="23b"^{}="23a"
    \ar@/^3mm/[r]_{}="22b"^{}="22a" \ar@/^9mm/[r]_{}="21"
    \ar@{=>}"21";"22a"
    \ar@{=>}"22b";"23a"
    \ar@{=>}"23b";"24"
    & \quad
    \ar@/_9mm/[r]^{}="14" \ar@/_3mm/[r]_{}="13b"^{}="13a"
    \ar@/^3mm/[r]_{}="12b"^{}="12a" \ar@/^9mm/[r]_{}="11"
    \ar@{=>}"11";"12a"
    \ar@{=>}"12b";"13a"
    \ar@{=>}"13b";"14"
    & \quad 
    \\ \\
  }}
\]
while the second corresponds to more general diagrams such as the following:
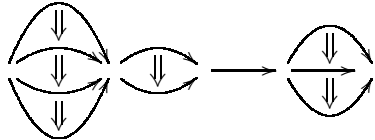
\[\vcenter{\xymatrix{
    \\ \quad
    \ar@/_9mm/[r]^{}="34" \ar@/_3mm/[r]_{}="33b"^{}="33a"
    \ar@/^3mm/[r]_{}="32b"^{}="32a" \ar@/^9mm/[r]_{}="31"
    \ar@{=>}"31";"32a"
    \ar@{=>}"32b";"33a"
    \ar@{=>}"33b";"34"
    & \quad
    \ar@/_3mm/[r]_{}="23b"^{}="23a"
    \ar@/^3mm/[r]_{}="22b"^{}="22a"
    \ar@{=>}"22b";"23a"
    & \ar[r]
    & 
    \ar@/_6mm/[r]^{}="13"
    \ar[r]_{}="12b"^{}="12a" \ar@/^6mm/[r]_{}="11"
    \ar@{=>}"11";"12a"
    \ar@{=>}"12b";"13"
    & \quad 
    \\ \\
  }}
\]
A diagram of the latter form (i.e.\ a 2-cell in $\Gm_{\cG(\cW)}(\cG(\Gm_\cW)(A))$) is called a \textbf{2-dimensional globular pasting diagram (2-pd)} in $A$.
Note that the 3-cells of $\Gm_{\cG(\cW)}(\cG(\Gm_\cW)(A))$ are not general ``3-dimensional globular pasting diagrams'' in $A$ but merely ``morphisms of 2-dimensional ones,'' i.e.\ diagrams such as
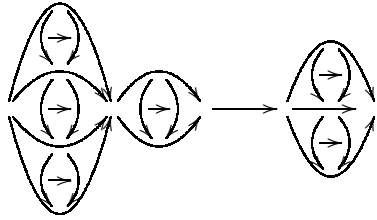
\[\vcenter{\xymatrix{
    \\ \quad
    \ar@/_14mm/[r]^{}="34"
    \ar@/_5mm/[r]_{}="33b"^{}="33a"
    \ar@/^5mm/[r]_{}="32b"^{}="32a"
    \ar@/^14mm/[r]_{}="31"
    \ar@/_2.5mm/ "31";"32a" ^{}="311"
    \ar@/^2.5mm/ "31";"32a" _{}="312"
    \ar@/_2.5mm/ "32b";"33a"  ^{}="321"
    \ar@/^2.5mm/ "32b";"33a" _{}="322"
    \ar@/_2.5mm/ "33b";"34" ^{}="331"
    \ar@/^2.5mm/ "33b";"34" _{}="332"
    \ar "311";"312"
    \ar "321";"322"
    \ar "331";"332"
    & \quad
    \ar@/_5mm/[r]_{}="23b"^{}="23a"
    \ar@/^5mm/[r]_{}="22b"^{}="22a"
    \ar@/_2.5mm/"22b";"23a" ^{}="221"
    \ar@/^2.5mm/"22b";"23a" _{}="222"
    \ar "221";"222"
    & 
    \ar[r]
    & 
    \ar@/_9mm/[r]^{}="13"
    \ar[r]_{}="12b"^{}="12a"
    \ar@/^9mm/[r]_{}="11"
    \ar@/_2.5mm/"11";"12a" ^{}="111"
    \ar@/^2.5mm/"11";"12a" _{}="112"
    \ar@/_2.5mm/"12b";"13" ^{}="121"
    \ar@/^2.5mm/"12b";"13" _{}="122"
    \ar "111";"112"
    \ar "121";"122"
    & \quad 
    \\ \\
  }}
\]
where each 2-cell in a 2-pd is replaced by a single 3-cell.

The other difference between~\eqref{eq:rectpdsum} and~\eqref{eq:pdsum} is that the ordering of the factors is different: in~\eqref{eq:rectpdsum} we compose horizontally and then vertically, while in~\eqref{eq:pdsum} we compose vertically and then horizontally.
Since rectangular 2-pds are a special case of general ones, and since \cW is symmetric, there is an obvious map from~\eqref{eq:rectpdsum} to~\eqref{eq:pdsum}, and this is the distributive law \lambda.

By the general theory of distributive laws, we can now conclude that
\begin{enumerate}
\item $\Gm_{\cG(\cW)}$ lifts to a monad $\widetilde{\Gm_{\cG(\cW)}}$ on
  the \cV-category of $\cG(\Gm_\cW)$-algebras, and
\item the composite functor $\Gm_{\cG(\cW)}\circ \cG(\Gm_\cW)$ on
  $\cG(\cG(\cW))$ has the structure of a monad, whose algebras are the
  same as those of $\widetilde{\Gm_{\cG(\cW)}}$.
\end{enumerate}

The description above of $\Gm_{\cG(\cW)}\circ \cG(\Gm_\cW)$ makes it ``obvious'' that when $\cW=\cCat$, its algebras should be strict 3-categories, but for purposes of generalization in \S\ref{sec:enrich-bicat} we prefer to deduce that from a general analysis.
For this we require two more observations.

The first is essentially~\cite[Lemma 2.4]{bcw:algop-enr-ii}.
Recall that for a monad $t\colon A \to A$ in a 2-category \cK, a (strict) \emph{Eilenberg-Moore object (EM-object)} for $t$ is an object $A^t$ together with an isomorphism of categories
\[ \cK(X, A^t) \cong \cK(X,A)^{\cK(X,t)} \]
natural in $X$, where $\cK(X,A)^{\cK(X,t)}$ denotes the usual Eilenberg-Moore category (category of algebras) for the ordinary monad $\cK(X,t)$ on the category $\cK(X,A)$.
Unsurprisingly, EM-objects in \cCat (or, more generally, \cat\cV) are simply ordinary Eilenberg-Moore (\cV-)categories.

\begin{lem}\label{thm:id-em}
  The 2-category \vcatidc admits the construction of Eilenberg-Moore objects, which are preserved by the forgetful 2-functor
  $ \vcatidc \to \CAT\cV $
  and also by the 2-functor
  $ \cG\colon \vcatidc\to\vcatidc $.
\end{lem}
\begin{proof}
  On the one hand, it is shown in~\cite{street:ftm} that EM-objects can be described as a certain kind of \emph{lax limit}.
  On the other hand, it is proven in~\cite{lack:lim-lax} that for any 2-monad $S$ on a 2-category \cK, the forgetful functor $\Alg{S}_c \to \cK$ creates all lax limits, where $\Alg{S}_c$ denotes the 2-category of strict $S$-algebras and \emph{colax} $S$-morphisms.
  Modulo size considerations (which can be dealt with as in~\cite{dl:lim-smallfr}), there is a 2-monad $S$ on \CAT{\cV} such that $\Alg{S}_c = \vcatidc$.
  Thus, \vcatidc admits EM-objects constructed as in \CAT\cV.
  To avoid size questions, we can apply this argument only to the monoidal structure, and observe separately that the category of algebras for any monad inherits any colimits preserved by the monad (which is a special case of the theorem of~\cite{lack:lim-lax}, but also easy to prove directly).

  It remains to show that \cG preserves EM-objects.
  But for any monad $T$ in \vcatidc, the unit of $\cG(T)$ is bijective on objects, so the algebra structure of any $\cG(T)$-algebra must also be bijective on objects.
  It follows that a $\cG(T)$-algebra
  structure on a \cW-enriched graph is just a $T$-algebra structure on
  each hom-object.  That is to say, a $\cG(T)$-algebra is the same as
  a graph enriched in $T$-algebras, i.e.\ $\cG$ preserves EM-objects.
\end{proof}

This implies two things.  Firstly, since $\Gm_\cW$ is a monad in
\vcatidc, its \cV-category of algebras, namely $\cat{\cW}$, is also a
\ten-distributive symmetric monoidal \cV-category.  Secondly, the
\cV-category of $\cG(\Gm_\cW)$-algebras is equivalent to the
\cV-category $\cG(\cat{\cW})$ of graphs enriched in \cW-categories.

The next observation is essentially~\cite[Corollary
6.11]{bcw:algop-enr-ii}.

\begin{lem}\label{thm:gph-cat}
  The induced monad $\widetilde{\Gm_{\cG(\cW)}}$ on $\cG(\cat{\cW})$ is
  isomorphic to $\Gm_{\cat{\cW}}$.
\end{lem}
\begin{proof}%[Sketch of proof]
  Let $A$ be a (\cat{\cW})-enriched graph, i.e. a
  $\cG(\Gm_\cW)$-algebra.  Thus it is a \cW-enriched 2-graph, as
  above, together with, for each $x,y\in A_0$, a \cW-category
  structure whose objects are edges $f,g\colon x\to y$ in $A_1$ and
  whose morphism-objects are the $A(f,g)$.

  By definition, $\widetilde{\Gm_{\cG(\cW)}}$ applies $\Gm_{\cG(\cW)}$
  to underlying objects in $\cG(\cG(\cW))$ and equips the result with
  a $\cG(\Gm_\cW)$-algebra structure specified by \lambda.  Thus the
  underlying directed graph of $\widetilde{\Gm_{\cG(\cW)}}(A)$ is the
  free category on $A_1\toto A_0$, and we have
  \[ \widetilde{\Gm_{\cG(\cW)}}(A)\big((f_n,\dots,f_1),(g_m,\dots,g_1)\big) =
  \begin{cases}
    A(f_n,g_n) \ten \cdots \ten A(f_1,g_1) & \quad \text{if }n=m\\
    \emptyset & \quad \text{if }n\neq m.
  \end{cases}
  \]
  with the local \cW-category structure given by
  \begin{multline*}
    \Big(A(g_n,h_n) \ten \cdots \ten A(g_1,h_1)\Big) \ten
    \Big(A(f_n,g_n) \ten \cdots \ten A(f_1,g_1)\Big)\\
    \too[\iso]
    \Big(A(g_n,h_n) \ten A(f_n,g_n)\Big) \ten \cdots\ten
    \Big(A(g_1,h_1) \ten A(f_1,g_1)\Big)\\
    \too A(f_n,h_n) \ten\cdots \ten A(f_1,h_1).
  \end{multline*}
  On the other hand, $\Gm_{\cat{\cW}}$ is built in the same way as
  $\Gm_{\cG(\cW)}$, but using sums and tensor products in \cat{\cW}
  instead of $\cG(\cW)$.  But sums and tensor products in \cat{\cW}
  are created in $\cG(\cW)$, with a \cW-category structure induced
  from the colax monoidal structure of $\Gm_\cW$, and this gives exactly the same structure maps as above.
\end{proof}

It follows that $\widetilde{\Gm_{\cG(\cW)}}$-algebras can be identified
with categories enriched in \cW-categories, i.e.\ with \textbf{\cW-enriched 2-categories}.
We conclude:

\begin{thm}\label{thm:alg-gg}
  If \cW is a \ten-distributive symmetric monoidal \cV-category, then
  there is a \cV-monad $\Gm_{\cG(\cW)}\circ \cG(\Gm_\cW)$ on the
  \cV-category $\cG(\cG(\cW))$, whose \cV-category of algebras
  consists of \cW-enriched 2-categories.\qed
\end{thm}

We will write $T_\cW$ for this monad $\Gm_{\cG(\cW)}\circ
\cG(\Gm_\cW)$.  By its explicit description given above, we see that
it equips its algebras with a direct way to compose any
2-dimensional globular pasting diagram, as we would expect for an
``unbiased'' monadic presentation of (enriched) 2-categories.

\section{Iconic tricategories}
\label{sec:enrich-bicat}

We now specialize again to the case $\cV=\cCat$.  Thus, for any
\ten-distributive sym\-metric monoidal 2-category \cW, we have a 2-monad
$T_\cW$ on \cW-enriched 2-graphs whose strict algebras are
\cW-enriched 2-categories.  In particular, if $\cW=\cCat$ as well,
then strict $T_{\cCat}$-algebras are strict 3-categories.  The
morphisms of $\Alg{T_{\cCat}}_s$ are of course strict 3-functors.
We follow~\cite{gg:ldstr-tricat} in calling its 2-cells \textbf{ico-icons};
they can be identified with ``oplax tritransformations'' whose 1- and 2-morphism components are identities.
(This 2-category $\Alg{T_{\cCat}}_s$ of 3-categories, 3-functors, and ico-icons is the next level of a ``low-dimensional category of higher-dimensional categories.'')

However, we can now also consider pseudo $T_{\cCat}$-algebras.  From
the description of $T_\cW = \Gm_{\cG(\cW)} \circ \cG(\Gm_\cW)$ in the previous section, we see that a
pseudo $T_{\cCat}$-algebra is a \cCat-enriched 2-graph with the
structure of a category on its underlying directed graph, together
with basic operations for composing any 2-pd, which are functorial and satisfy the appropriate laws up to invertible 3-cells.

This looks like some sort of tricategory, but to describe it in a more familiar way, we need to unravel the relationship between pseudoalgebras and distributive laws.
A natural context for this involves pseudomonads.
By a \textbf{pseudomonad} we will mean a strict 2-functor $T$ equipped with pseudo natural transformations $\mu\colon T^2\to T$ and $\eta\colon \mathrm{Id} \to T$ which satisfy the monad laws up to coherent invertible modifications.
(This is not the only possible weakening of the notion of 2-monad, of course---one could require $\mu$ and $\eta$ to be strict, or allow $T$ to be only a pseudofunctor---but it is the most convenient for our purposes.)
For any pseudomonad $T$ we can define the 2-category $\PsAlg{T}$ in a straightforward way.
Similarly, a \textbf{pseudo distributive law} between pseudomonads $T$ and $S$ is a pseudo natural transformation $T S \to S T$ which satisfies the distributive law axioms up to coherent invertible modifications.

Of course, any strict 2-monad is also a pseudomonad, and likewise any strict \cCat-enriched distributive law is a pseudo distributive law.
This applies in particular to our 2-monads $\cG(\Gm_\cW)$ and $\Gm_{\cG(\cW)}$ and our distributive law
\[\lambda\colon \cG(\Gm_\cW) \circ \Gm_{\cG(\cW)} \too 
\Gm_{\cG(\cW)}\circ \cG(\Gm_\cW).
\]
It is shown in~\cite{marmolejo-psmonads,marmolejo-psmonads-ii} that pseudo distributive laws between pseudomonads correspond to liftings to 2-categories of pseudoalgebras, just as in the strict case.
In our situation, this implies that for any \cW,
\begin{enumerate}
\item $\Gm_{\cG(\cW)}$ lifts to a pseudomonad
  $\widetilde{\Gm_{\cG(\cW)}}$ on $\PsAlg{\cG(\Gm_\cW)}$, and
\item the composite functor $T_\cW = \Gm_{\cG(\cW)}\circ \cG(\Gm_\cW)$
  has the structure of a pseudomonad, such that $\PsAlg{T_\cW}$ is
  equivalent to $\PsAlg{\widetilde{\Gm_{\cG(\cW)}}}$.
\end{enumerate}
Moreover, since $\Gm_{\cG(\cW)}$ and $\cG(\Gm_\cW)$ are strict
2-monads and the distributive law \lambda\ is strict, the pseudomonads
$\widetilde{\Gm_{\cG(\cW)}}$ and $T_\cW$ are also strict 2-monads, and
$T_\cW$ is the same 2-monad to which we gave that name in the previous
section.

We now require the analogues of Lemmas \ref{thm:id-em} and \ref{thm:gph-cat} for pseudoalgebras.
The natural context in which to prove these is that of Gray-categories.
Recall from~\cite{gps:tricats} that \cGRAY denotes the category of strict 2-categories and strict 2-functors, equipped with the closed symmetric monoidal structure whose internal-hom $[C,D]$ is the 2-category of strict 2-functors, \emph{pseudo} natural transformations, and modifications from $C$ to $D$.
A \textbf{Gray-category} is a \cGRAY-enriched category, which can be considered as a semi-strict form of tricategory.
The prototypical Gray-category is of course \cGRAY itself, which (as a Gray-category) consists of strict 2-categories, strict 2-functors, pseudonatural transformations, and modifications.

Since the notion of pseudomonad we are using involves strict 2-functors and pseudonatural transformations, it can be defined entirely within the Gray-category \cGRAY.
By mimicking this definition we can define pseudomonads inside any Gray-category (see for instance~\cite{marmolejo-psmonads}).
One can also define \emph{objects of pseudoalgebras}, which generalize $\PsAlg{T}$ in the same way that EM-objects generalize Eilenberg-Moore categories; see~\cite{lack:psmonads}, where these are also exhibited as a certain kind of \cGRAY-weighted limit.

We can now state and prove a pseudo version of \autoref{thm:id-em}.
Since the technology of pseudomonads is less well-developed, for simplicity we now restrict to the case when \cW is cartesian monoidal.
(Of course, our primary interest is in the case $\cW=\cCat$.)
Let \grayidc denote the Gray-category of $\times$-distributive cartesian monoidal 2-categories, 2-functors preserving small sums, pseudonatural transformations, and modifications.
Once we verify that \cG acts on pseudonatural transformations and modifications, we have a Gray-functor $\cG\colon \grayidc \to \grayidc$.

\begin{lem}
  The Gray-category \grayidc has objects of pseudoalgebras, and they are preserved by the forgetful Gray-functor $\grayidc\to\cGRAY$ and by the Gray-functor \cG.
\end{lem}
\begin{proof}
  The 2-category of pseudoalgebras for a pseudomonad inherits
  finite products from the base 2-category, along with any sums that
  the pseudomonad preserves.  (The statement about products
  essentially follows from~\cite[2.1]{bkp:2dmonads}, and both
  statements are special cases of the general results
  of~\cite{ls:limlax}; but both are also easy to check directly.)
  Hence if $T$ is a pseudomonad on \cK in \grayidc, then $\PsAlg{T}$ is again
  \times-distributive, with structure created in \cK.
  This implies that \grayidc has objects of pseudoalgebras preserved by its forgetful functor to $\cGRAY$.

  Now since isomorphic maps in $\cG(\cW)$ must be equal on objects, as
  in \autoref{thm:id-em} we conclude that a pseudo $\cG(T)$-algebra
  structure must be given locally.  Thus pseudo $\cG(T)$-algebras are
  just pseudo-$T$-algebra--enriched graphs, i.e.\ \cG preserves objects of pseudoalgebras as well.
\end{proof}

As before, this implies that the 2-category \pscat{\cW} of pseudo
$\Gm_\cW$-algebras is again a $\times$-distributive cartesian monoidal
2-category, and that the 2-category of pseudo $\cG(\Gm_\cW)$-algebras
is equivalent to the 2-category $\cG(\pscat{\cW})$ of graphs enriched
in pseudo \cW-categories.

\begin{lem}
  The 2-monad $\widetilde{\Gm_{\cG(\cW)}}$ on $\cG(\pscat{\cW})$ is
  isomorphic to $\Gm_{\pscat{\cW}}$.
\end{lem}
\begin{proof}
  Just like the proof of \autoref{thm:gph-cat}.
\end{proof}

Thus pseudo $\widetilde{\Gm_{\cG(\cW)}}$-algebras can be
identified with pseudo ($\pscat{\cW}$)-categories, i.e.\ we have
$\PsAlg{\widetilde{\Gm_{\cG(\cW)}}} \simeq \pscat{(\pscat{\cW})}$.
Combining this with the facts about distributive laws cited
previously, we have:

\begin{thm}\label{thm:psalg-bb}
  The 2-category $\PsAlg{T_\cW}$ of pseudoalgebras for the 2-monad
  $T_\cW$ on $\cG(\cG(\cW))$ is 2-equivalent to
  $\pscat{(\pscat{\cW})}$.\qed
\end{thm}

In particular, we have $\PsAlg{T_{\cCat}}\simeq \pscat{\cBicat}$.
(Recall that for us, \cBicat denotes the 2-category of bicategories, pseudofunctors, and icons.)
Explicitly, a (biased) pseudo \cBicat-category $A$ consists of:
\begin{enumerate}
\item A set of objects.
\item For each pair of objects $x,y$, a bicategory $A(x,y)$.
\item For each $x$, a pseudofunctor $1\to A(x,y)$.
\item For each $x,y,z$, a pseudofunctor $A(y,z)\times A(x,y) \to
  A(x,z)$.
\item For each $x,y,z,w$, an invertible icon
  \[\vcenter{\xymatrix{A(z,w)\times A(y,z)\times A(x,y) \ar[r]\ar[d]
      \drtwocell\omit{\iso}&
      A(z,w)\times A(x,z) \ar[d]\\
      A(y,w)\times A(x,y)\ar[r] & A(x,w) }}
  \]
\item For each $x,y$, invertible icons
  \[\vcenter{\xymatrix{
      &  \mathllap{A(y,y)}\times \mathrlap{A(x,y)} \ar[dr]
      &&
      &  \mathllap{A(x,y)}\times \mathrlap{A(x,x)} \ar[dr]\\
      A(x,y) \ar[ur] \ar@{=}[rr]
      \rrlowertwocell\omit{<-3>\iso}
      && A(x,y)
      &
      A(x,y) \ar[ur] \ar@{=}[rr]
      \rrlowertwocell\omit{<-3>\iso}
      && A(x,y)
    }}
  \]
\item These icons satisfy the pentagon and unit axioms for a
  bicategory.
\end{enumerate}
Comparing this to the definition of a tricategory
from~\cite{gps:tricats}, we see that the pseudonatural equivalences
for associativity and units have been replaced by invertible icons,
and the modifications $\pi$, $\mu$, $\lambda$, and $\rho$ have been
replaced by axioms.  However, invertible icons can be identified with
pseudonatural transformations whose 1-cell components are identities
(by composing with unit constraints, if necessary).  Under this
translation, the assertion that these icons satisfy the bicategory
axioms translates to the assertion that we have modifications $\pi$,
$\mu$, $\lambda$, and $\rho$ whose components are constraint 2-cells.
By coherence for bicategories, these constraints are unique, and
satisfy any axiom one might ask them to, including in particular the
tricategory axioms.  This suggests the following definition and
proposition.

\begin{defn}
  A tricategory is \textbf{iconic} if the 1-cell components of its
  associativity and unit constraints are identities, and the
  components of its modifications $\pi$, $\mu$, $\lambda$, and $\rho$
  are the uniquely specified constraint cells.
\end{defn}

\begin{prop}
  To give a \cCat-enriched 2-graph the structure of a pseudo
  \cBicat-category (i.e.\ a pseudo $T_{\cCat}$-algebra) is the same
  as to give it the structure of an iconic tricategory.  Moreover,
  under this equivalence, strict $T_{\cCat}$-algebras correspond
  precisely to strict 3-categories.\qed
\end{prop}

Thus, the 2-category $\PsAlg{T_{\cCat}}$ consists of iconic tricategories, ``iconic functors,'' and ico-icons.

\begin{rmk}
  In~\cite{gg:ldstr-tricat}, Garner and Gurski constructed a \emph{bicategory} whose objects and morphisms are arbitrary tricategories, arbitrary lax functors between them, and an appropriate sort of ico-icon.
  The objects and morphisms of $\PsAlg{T_{\cCat}}$ are rather more restricted, but one can construct a functor from $\PsAlg{T_{\cCat}}$ to their bicategory.
\end{rmk}

\begin{rmk}
  Every Gray-category is iconic when regarded as a tricategory, since composition of 1-cells in a Gray-category is strictly associative and unital.
  In particular, since every tricategory is triequivalent to a Gray-category, every tricategory is triequivalent to an iconic one.
\end{rmk}

We can now prove the main theorem in two different ways.

\begin{thm}\label{thm:main}
  Not every pseudo $T_{\cCat}$-algebra is equivalent to a strict one.
\end{thm}
\begin{proof}[First proof]
  The same arguments as for pseudo $T_{\cCat}$-algebras show that any
  pseudo $T_{\cCat}$-morphism induces an ``iconic'' functor of
  tricategories (one whose constraints have identity 1-cell components and whose higher constraints are unique bicategory coherence data).
  Moreover, any equivalence in $\PsAlg{T_{\cCat}}$ is
  bijective on 0-cells and 1-cells and locally locally an equivalence (i.e.\ an equivalence on hom-categories of hom-bicategories),
  hence induces a triequivalence of iconic tricategories.
  But any Gray-category is an iconic
  tricategory, hence arises from a pseudo $T_{\cCat}$-algebra, and we
  know that not every Gray-category is triequivalent to a strict
  3-category.  Therefore, not every pseudo $T_{\cCat}$-algebra can be
  equivalent to a strict one.
\end{proof}

We can also give a proof not using any tricategories, by restricting
to doubly-degenerate objects (those with exactly one 0-cell and one
1-cell).

\begin{proof}[Second proof]
  Since equivalences in $\cG(\cG(\cCat))$ are bijective on 0- and
  1-cells, doubly-degenerate objects are closed under equivalences.
  By \autoref{thm:psalg-bb} and \autoref{thm:psmon}, the 2-category of
  doubly-degenerate pseudo $T_{\cCat}$-algebras is equivalent to the
  2-category $\cPsmon(\cPsmon(\cCat))$, which
  by~\cite[\S5]{js:brd-tensor} is equivalent to the 2-category of
  braided monoidal categories.  However, since strict
  $T_{\cCat}$-algebras are strict 3-categories, doubly-degenerate ones
  can be identified with strictly-symmetric strict monoidal
  categories.  Thus, any non-symmetric braided monoidal category (such
  as, for example, the braid category) induces a pseudo
  $T_{\cCat}$-algebra that is not equivalent to a strict one.
\end{proof}

\begin{rmk}
  Note that $\cG(\cG(\cCat))$ is locally finitely presentable and
  $T_{\cCat}$ is finitary.  Thus, the 2-monad $T_{\cCat}$ is otherwise
  as well-behaved as one could wish, but it still violates the
  strictification theorem.
\end{rmk}

\begin{rmk}
  Given the second proof of \autoref{thm:main}, one might wonder whether the introduction of pseudo enriched categories was really necessary, or whether pseudomonoids would suffice.
  However, in order for $T_\cW$ to be a 2-monad rather than a pseudomonad, we needed a \emph{strict} distributive law \lambda.
  This, in turn, requires the fact that $\Gm_\cW$ preserves sums, which is not true of the free monoid monad.
\end{rmk}

\begin{rmk}
  It seems that iconic tricategories may be of independent interest,
  since they are more general than Gray-categories, yet still have a
  purely 2-categorical description as pseudo \cBicat-categories or
  pseudo $T_{\cCat}$-algebras.  Moreover, many naturally occurring
  tricategories seem to be iconic, including even the ``prototypical''
  tricategory of bicategories, pseudofunctors, pseudonatural
  transformations, and modifications.
\end{rmk}

\begin{rmk}
  In a sense, 2-monads such as $\Gm_{\cCat}$ and $T_{\cCat}$ bring together two distinct threads within coherence theory: the 2-categorical and the higher-dimensional.
  It would be interesting to consider whether such a synthesis can also yield positive results.
  Power already observed in~\cite{power:coherence} that his general coherence theorem implies strictification for (unbiased) bicategories.
  One dimension up, there is a 2-monad on \cCat-enriched 2-graphs whose strict algebras are Gray-categories, and whose pseudoalgebras are again essentially the same as iconic tricategories.
  I do not know whether 2-categorical methods could be applied to this 2-monad to prove that any iconic tricategory is equivalent to a Gray-category.

  Alternatively, we could consider higher-dimensional monads: there is a Gray-monad on \cGRAY-enriched graphs whose strict algebras are Gray-categories, and whose pseudoalgebras (in the 3-categorical sense of~\cite{power:3dmonads}) are a type of unbiased ``cubical'' tricategory.
  Thus, if the theorems of~\cite{power:coherence,lack:codescent-coh} can be extended to Gray-monads, they would imply part of the coherence theorem for tricategories.
\end{rmk}

\bibliographystyle{alpha}
\bibliography{all,shulman}

\end{document}